\documentclass{amsart}
\usepackage{amsmath,amssymb}
\usepackage[all]{xy}
\numberwithin{equation}{section}
\theoremstyle{plain}

\newtheorem{introtheorem}{Theorem}
\newtheorem{theorem}{Theorem}[section]
\newtheorem{proposition}[theorem]{Proposition}
\newtheorem{lemma}[theorem]{Lemma}

\newtheorem{conjecture}[theorem]{Conjecture}

\newtheorem{question}[theorem]{Question}
\newtheorem*{proposition*}{Proposition}

\theoremstyle{definition}

\newtheorem{example}[theorem]{Example}

\theoremstyle{remark}
\newtheorem{remark}[theorem]{Remark}
\newtheorem{remarks}[theorem]{Remarks}

\newcommand{\exref}[1]{Example~\ref{#1}}

\def\N{{\mathbb N}}

\def\Q{{\mathbb Q}}
\def\C{{\mathbb C}}
\def\L{{\mathbb L}}

\def\map{\mathrm{map}}
\def\Der{\mathrm{Der}}

\def\cat0{\mathrm{cat}_0}

\def\dim{\mathrm{dim}}

\def\aut{\mathrm{aut}_1}

\def\B{B\mathrm{aut}_1}

\def\ad{\mathrm{ad}}

\begin{document}

\title[Gottlieb elements  and  the Sullivan model]
{The structuring effect of a Gottlieb element on the Sullivan model of a space}

\author{Gregory  Lupton}

\address{Department of Mathematics,
           Cleveland State University,
           Cleveland OH 44115}

\email{G.Lupton@csuohio.edu}

\author{Samuel Bruce Smith}

\address{Department of Mathematics,
   Saint Joseph's University,
   Philadelphia, PA 19131}

\email{smith@sju.edu}

\date{\today}

\keywords{Gottlieb elements, Sullivan minimal models, classifying space for fibrations,  derivations  }
\subjclass[2010]{Primary: 55P62 55R35; Secondary: 55R15}
\begin{abstract}
We   show a Gottlieb element in the rational homotopy of a simply connected space  $X$  implies a structural result  for the Sullivan minimal model, with different results depending on parity.  
 In the even-degree case,  we prove a rational Gottlieb element is a terminal homotopy element.  This fact allows us to  complete an   argument of Dupont to prove an even-degree Gottlieb element gives   a free factor in the rational cohomology of a formal space of finite type.     We apply the odd-degree result to     affirm  a special case of the $2N$-conjecture on Gottlieb elements of a finite complex.   We combine our results to make a contribution to the realization problem for the classifying space $\B(X)$.  
We prove a simply connected space $X$ satisfying $\B(X_\Q) \simeq  S_\Q^{2n}$ must have infinite-dimensional rational homotopy and vanishing rational Gottlieb elements above degree $2n-1$ for $n= 1, 2, 3.$  

\end{abstract}   

\maketitle

\section{Introduction}
Let $X$ be   simply connected  and  of finite CW type.  A homotopy class $\alpha \in \pi_n(X)$ is a {\em Gottlieb element}  if the map   of the wedge $(\alpha  \mid \!  1_X ) \colon  S^n   \vee X \to X$  extends to    a map   of the product $F \colon S^n \times X  \to X$.   The definition directly implies the vanishing of the Whitehead product of $\alpha$   with any $\beta \in \pi_m(X)$.  Amongst  many nice results,  Gottlieb proved that    an even-degree  Gottlieb element is in the kernel of the mod $p$ Hurewicz homomorphism   for    any  prime $p$  not dividing the Euler characteristic   of $X$    \cite[Th.4.4]{Got}.  
     
 Gottlieb elements have a natural    description in   rational homotopy theory which we recall,  briefly,  here.   Our general reference  for rational homotopy theory  is the text  \cite{FHT}.     A space $X,$ as hypothesized,   has a Sullivan  minimal model which is    a free DG algebra $(\land V, d)$ over $\Q$   with each $V^n$  finite-dimensional.  The differential $d$ has  image in the decomposable elements of ${\land} V$.       The affiliated map $F$ for a Gottlieb class $\alpha \in \pi_n(X)$ induces a   map $F^* \colon (\land V, d) \to (\Q(u), 0)  \otimes   (\land V, d)$ of DG algebras.  Here  $u$ is of degree $n$ and    $(\Q(u), 0) \cong H^*(S^{n}; \Q)$      is    a non-minimal Sullivan model for $S^n$ with trivial differential. 
Since $F$ extends $(\alpha \! \mid \! 1_X )$  we have $F^*(v) = u$ where $v \in V^n$ is dual to $\alpha$ under Sullivan's isomorphism  $V^n \cong \mathrm{Hom}(\pi_n(X),  \Q)$ \cite[Th.15.11]{FHT}.    Further, writing $F^*(\chi) = \chi  + u\theta(\chi)$ for $\chi \in \land V,$ the linear map    $\theta$ so-defined  is   a    {\em derivation}  of $\land V$, $\theta(\chi_1\chi_2)=\theta(\chi_1)\chi_2 + (-1)^{|\chi_1|n} \chi_1\theta(\chi_2),$  lowering degrees by $n,$ of ({\em degree $n$})  and $\theta$ is a {\em derivation cycle},  $d \theta - (-1)^n\theta d = 0$.     
We say    $v \in V^n$ is a {\em Gottlieb element} for the Sullivan minimal model $(\land V, d)$   if there exists a derivation cycle $\theta$ of $\land V$ of degree $n$ with $\theta(v) =1.$   

In \cite[Lem.1.1]{Hal88},  Halperin showed that the derivation cycle $\theta$ associated to a Gottlieb element $v \in V^n$  that is  a cycle,  $dv = 0$,   induces  a change of basis for the Sullivan model $(\land V, d)$ and a resulting DG algebra factorization: $$(\land V, d) \cong (\land(v), 0) \otimes (\land V' ,d').$$ 
 Recall the space of cycles $Z(V) \subseteq V$    may be identified as the dual of the image of the rational Hurewicz homomorphism. The result thus represents a natural extension of Gottlieb's theorem,  mentioned above (cf.  \cite{Opr86}).       
 
 In this paper, we explore the structural consequence  of a Gottlieb element $v \in V^n$ in the general case,  when        $dv \neq 0$.    For $n$ even, we obtain  a surprising  result.  Say an element $v \in V^n$    is a {\em terminal element}   if 
  $v$  does not appear in a  differential $dw$ for any $w \in V$. 
 
\begin{introtheorem}  \label{terminal}  An even-degree   Gottlieb element $v \in V^{2n}$ in a Sullivan minimal model $(\land V, d)$  is  a terminal element.  
\end{introtheorem}

% Theorem \ref{terminal}  applies to complete a proof of    Dupont on the even-degree Gottlieb elements of a formal space.  
We apply Theorem \ref{terminal} to an old problem on the location of Gottlieb elements for a formal space $X$. We recall   that a formal space $X$  has a \emph{bigraded} Sullivan minimal model   $(\land V, d)$ with the generators carrying a second, or lower, grading
$V = \bigoplus_{i \geq 0} V_i$ that extends multiplicatively to the whole of $\land V$. The differential satisfies $d(V_0) = 0$ and $d(V_i) \subseteq \left( \land V \right)_{i-1}$ for $i \geq 1$ (see \cite{HS}).  The following conjecture is attributed to C. Jacobsson.    
\begin{conjecture}\label{conj: Jacobsson}
If $X$ is a formal space with bigraded Sullivan minimal model $(\land V, d)$, then all Gottlieb elements   are contained in $V_0 \oplus V_1$.  
\end{conjecture}

In  \cite{Dupont}, Dupont initiated work on Conjecture \ref{conj: Jacobsson} and enunciated several results on this problem. However,  this paper was never published and some arguments appear to be incomplete.    We  here reproduce   Dupont's argument in  the even-degree   case  of  \cite[Pro.4]{Dupont}  and use Theorem \ref{terminal}   to complete the proof.  
 
 \begin{introtheorem}  
    {\em (cf. \cite[Pro.4]{Dupont})}    \label{J} 
Let $X$ be a formal space with finitely generated rational cohomology and suppose $y \in V^{2n}$ is an even-degree  Gottlieb element in the   Sullivan minimal model $(\land V, d)$ for $X$.  Then $dy = 0$ and  there  is a DG algebra isomorphism $$(\land V, d) \cong (\land(y), 0)  \otimes  (\land V', d').$$    \end{introtheorem}

The mapping theorem for rational L.S. category (\cite[Th.28.6]{FHT}) implies strong constraints on the rational Gottlieb elements of a   space $X$ of finite L.S. category.  The  even-degree Gottlieb elements  vanish in this case  and the number   of independent odd-degree Gottlieb elements is  bounded above by the rational L.S. category of $X$ \cite[Pro.29.8]{FHT}.     The   location of the odd-degree  Gottlieb elements  is the subject of the following open problem.      
\begin{conjecture}  {\em \cite[p.518]{FHT} } \label{2Nconjecture}  Let $X$ be space with $H^*(X; \Q)$ finite-dimensional.  Let  $N = \mathrm{max}\{n \mid H^n(X; \Q) \neq 0\}.$ 
 Then the rational Gottlieb elements for $X$   are of  degree $< 2N$.   \end{conjecture}

In Theorem \ref{odd} below,    we    show that a    Gottlieb element $v \in V^{2n+1}$   induces  a basis change  for $(\land V, d)$  after which  the appearance of $v$ in differentials $dw$  is explicitly constrained.    We apply this result to prove:  \begin{introtheorem} \label{2N}    Let $X$ be space with $H^*(X; \Q)$ finite-dimensional  and with top nontrivial degree $N$.  Let  $x \in V^{2n+1}$ be a Gottlieb element  in the Sullivan minimal model $(\land V, d)$ with $dx$ a monomial in the generators of $V.$ Then $2n+1 < 2N.$
\end{introtheorem}

 We apply our   results  in both the even and odd-degree cases  to another open problem in rational homotopy, concerning the classifying space $\B(X)$ for fibrations.  We recall that   $X,$  as hypothesized,  has a {\em universal fibration} 
$  X \to E_X \to \B(X)$   such that  any fibration of simply connected spaces with fibre $X$ is equivalent to a pullback by a map into the base     $\B(X)$, the {\em classifying space}.   (cf. \cite{May}). There is  an  H-equivalence $\Omega \B(X) \simeq \aut(X)$ with $\aut(X) = \map(X, X; 1)$.    
With our hypotheses on $X,$ the construction can be applied to the rationalization  $X_\Q$  of $X$ yielding   
$X_\Q \to E_{X_\Q} \to \B(X_\Q)$, the universal fibration with fibre $X_\Q$.    

The construction of  algebraic  models for $\B(X_\Q)$  is the   classical work  of many authors (cf. \cite[Ch.7]{Tanre}) and an area of continued interest. Recent advances provide models  in the non simply-connected case  \cite{BZ, FFM}.    
We   make primary use here of   the first  and simplest description of a model,   due to Sullivan.  Write    $\Der_n(\land V)$ for the space of degree $n$ derivations of the Sullivan minimal model $(\land V, d)$ for $n \geq 1$     with the graded commutator bracket  $[\theta, \varphi] = \theta \circ \varphi + (-1)^{|\theta||\varphi|}\varphi \circ \theta$ and  differential $D(\theta) = [d, \theta]. $    Sullivan's identity \cite[Sec.11]{Su} is an isomorphism of connected, graded Lie algebras:      \begin{equation} \label{eqS} \pi_*(\Omega \B(X_ \Q))   \cong H_*(\Der(\land V), D). \end{equation} 
The following question is  due to  M. Schlessinger.    

\begin{question}\label{ques: Baut}  \em{\cite[p.517]{FHT}}
Is every  simply connected space  $Y$ realized, rationally,  as a classifying space,    in the sense that there is  some simply connected space $X$    such that   $\B(X_\Q) \simeq  Y_\Q?$ 
\end{question}

Question \ref{ques: Baut}  suggests a  realization problem  for classifying spaces.     Given a space $Y$, the problem is to  either  construct  $X$ with $\B(X_\Q) \simeq  Y_\Q$ or, alternately, to prove no such space $X$  exists.  The identity $K(\Q, n) \simeq \B(K(\Q, n-1))$ shows    Eilenberg-Mac Lane spaces are   realized in this sense, at least  for $n \geq 2,$   Deciding whether an arbitrary product $K(\Q, m) \times K(\Q, n)$  can be so realized is already challenging   (see \cite[Th.1]{LS2016} for one class of examples).       

In previous work,  we have  proved several low-dimensional  rational types  $Y$ cannot be realized as  $\B(X)$  when $X$ is restricted to have finite-dimensional rational homotopy ($X$ is {\em $\pi$-finite}).  We have proved this result for $Y =   S^{2n}$ for $n = 1, 2$   and for $Y= \C P^n$ for $n = 2, 3, 4$;  namely,      if $\B(X_\Q) \simeq  Y_\Q$  for any such  $Y$ then  $X$ must be  $\pi$-infinite  \cite{LS2016, LS2020}.  We extend and sharpen  these results with the following.    \begin{introtheorem}\label{introthm: S4} 
Let $X$  satisfy  $\B (X_\Q) \simeq  S_\Q^{2n}$ for $n = 1, 2, 3$.   Then  $X$ is $\pi$-infinite with vanishing rational Gottlieb elements   above degree $2n-1.$  
    \end{introtheorem}

Our  proof of Theorem \ref{introthm: S4}  streamlines   the argument given in \cite[Th.3]{LS2016} and extends it to include the case $S^6$.   
The  advance  over  our previous work    is  that here we do not start   by assuming $X$ is $\pi$-finite with $\B(X_\Q) \simeq  S_\Q^{2n}$.  Instead,  we give  a  constraint on the $\pi$-infinite spaces $X$ that could satisfy this identity.  
 
The paper is organized as follows.   In Section \ref{sec2}, we show   a  rational Gottlieb element induces a basis change for a Sullivan minimal model.    We deduce   Theorems \ref{J} and   \ref{2N}  as consequences in Section \ref{sec3}.       In Section \ref{sec4}, we  give a structure result for a Sullivan minimal model having   Gottlieb elements of both parities in   the scenario that arises in the realization problem   for $Y =S^{2n}$. We apply  this result to   prove  Theorem \ref{introthm: S4}.     In Section \ref{sec5},  we give    a further example suggesting a negative answer to Question \ref{ques: Baut}.   We prove a space $X$ with $\B(X_\Q) \simeq (S^3 \vee \cdots \vee S^3)_\Q$ must have vanishing rational Gottlieb elements.

 \section{Gottlieb Elements and Basis Change for Sullivan Models}
 \label{sec2}
 
 In this section, we  observe  that a Gottlieb element $v \in V^{n}$  induces a {\em change of basis isomorphism}  for the Sullivan minimal model $(\land V, d)$.       The idea is as follows.  Start with an  automorphism    $\phi \colon \land V \to \land V $ of  graded algebras.  Then $\phi$  induces a  new differential $d'$ on $\land V$ given by  $d' = \phi^{-1}\circ d \circ \phi$.  The map 
$$\phi \colon (\land V, d') \to (\land V, d).$$
is, tautologically,    a DG algebra isomorphism. Also note that, since $d$ is decomposable,   $d'$ is decomposable as well and $\phi$ is an isomorphism of minimal DG algebras.  

  A  derivation $\theta$ of $(\land V, d)$ induces a derivation  $\theta' = \phi^{-1}\circ \theta \circ \phi$  of $\land V$. If $\theta$ is a derivation cycle, i.e.,   $[d, \theta] = 0,$ then we see 
$$[d', \theta'] = \big[ \phi^{-1}\circ d \circ \phi, \phi^{-1}\circ \theta \circ \phi\big] = \phi^{-1}\circ [d, \theta] \circ \phi = 0.$$

We focus first on the structuring effect of   an odd-degree Gottlieb element    $x \in V^{2n+1}.$ Let  $\theta$ be of degree $2n+1$ and satisfy $[d, \theta] = 0$ and  
 $\theta(x) = 1$.    Write   $V = \langle x \rangle \oplus W$ for $W \subseteq V$ a complementary subspace to $\langle x \rangle$ in $V$. Define a linear map  $\phi \colon  V \to \land V$ by setting $\phi(x) = x$ and, for each $v \in W$, 
 \begin{equation}\label{eq: odd change of basis}
\phi(v) = v - x \theta(v).
 \end{equation} 
 Extend multiplicatively to a map of algebras $\phi \colon \land V \to \land V$.  It is easy to see that $\phi$ is an automorphism.    Notice further that 
$$
 \begin{aligned}
 \phi(v) \phi(v') &= \big(v - x \theta(v)\big)\big(v' - x \theta(v')\big) \\
 & = vv'  - x \theta(v) v' - (-1)^{|v|} \, xv \theta(v')\\
 & = vv' - x \theta(vv')
 \end{aligned}
$$
It follows that $\phi(\chi) = \chi - x \theta(\chi)$ for a general $\chi \in \land W$.  Also note that, again for a general $\chi \in \land W$,   we have 
$\phi(x \chi) = \phi(x)\phi(\chi) = x\big(\chi - x \theta(\chi)\big) = x\chi$.
As above, define $d' = \phi^{-1}\circ d \circ \phi$ and  $\theta' =   \phi^{-1}\circ \theta \circ \phi$  so that $[d', \theta'] = 0$.  We prove \begin{theorem}\label{odd}
Let $(\land V, d)$ be a Sullivan minimal model with a Gottlieb element $x \in V^{2n+1}$. Let $(\land V, d')$  be as constructed above and $\theta'$ the induced derivation cycle of degree $2n+1$ of the model $(\land V, d')$.   Write $V = \langle x \rangle \oplus W$.  Then:

\begin{itemize}  
\item[(a)]  $\theta'(x) = 1$   and $d'(x) = dx$ \\ 
\item[(b)]  For $\chi \in \land W$,  we have  $\theta'(\chi) = x\,\lambda(\chi) \hbox{\, for some \, } \lambda \in \Der_{4n+2}(\land W)$  \\
\item[(c)]   For $\chi \in \land W$,  
$$d'(\chi) =  -\theta'(\chi)\, dx + d'_0(\chi) = -x\lambda(\chi)dx + d'_0(\chi) $$
 for some $d'_0 \in \Der_{-1}(\land W)$ and $\lambda \in \Der_{4n+2}(\land W)$ as in (b).     \end{itemize}
\end{theorem}

\begin{proof}  Part (a) is immediate since $\phi(x) =x$ and further $\phi$ is the identity on elements of degree $\leq 2n$.  For (b) and (c), we begin with a general observation.  Suppose that $\psi \colon \land V \to \land V$ is a derivation.  Let $\chi \in \land W$.
 If we define linear maps on $\land W$ by
$$\psi(\chi) = x \psi_1(\chi) + \psi_0(\chi),$$
then both $\psi_1$ and $\psi_0$ are derivations of $\land W$ (of different degrees).  This is easy to check, as follows.  For a product of terms  $\chi, \chi' \in \land W$, write
 $$
 \begin{aligned}
\psi(\chi \chi') &= \psi(\chi) \chi' + (-1)^{|\psi| |\chi|} \chi \psi(\chi') \\
 & = \big(x\,\psi_1(\chi) + \psi_0(\chi)\big) \chi' + (-1)^{|\psi| |\chi|}  \chi \big(x\,\psi_1(\chi') + \psi_0(\chi')\big)\\
 & = x \big(\psi_1(\chi)\chi' +  (-1)^{|\psi| |\chi| + |\chi|}  \chi \psi_1(\chi')  \big) + 
 \big( \psi_0(\chi) \chi' + (-1)^{|\psi| |\chi|}  \chi \psi_0(\chi')\big)\\
 &= x \psi_1(\chi \chi') +   \psi_0(\chi \chi').
 \end{aligned}
 $$
In the last line, we used that $\psi_0$ has the same degree as $\psi$, whereas the degree of $\psi_1$ is of opposite parity to the degree of $\psi$, so we have
$|\psi| + 1 \equiv |\psi_0| \mod 2$.   

Applying this decomposition to  $\theta$ gives%
\begin{equation}\label{eq: theta 1-0}
\theta(\chi) = x\,\theta_1(\chi) + \theta_0(\chi).
\end{equation}
As just observed, this defines derivations $\theta_0$ and $\theta_1$ of $\land W$ with $\theta_0$ of degree $2n+1$ and $\theta_1$ of degree $4n+2$.   The isomorphism $\phi$ is given by
 \begin{equation}\label{eq: phi 1-0}
 \phi(\chi) = \chi - x\, \theta_0(\chi).
 \end{equation}
Furthermore, as already noted, we have $\phi(x \chi) = x\chi$, and hence also $\phi^{-1}(x \chi) = x\chi$.  Applying $\phi^{-1}$ to (\ref{eq: phi 1-0}), then, yields 
\begin{equation}\label{eq: phi inv 1-0}
\phi^{-1}(\chi) = \chi + x\, \theta_0(\chi).
\end{equation}
We  use the  notation from  \eqref{eq: theta 1-0}, \eqref{eq: phi 1-0}, and \eqref{eq: phi inv 1-0} in what follows.

\medskip

\noindent{}(b)  We calculate:
$$
 \begin{aligned} 
\theta'(\chi) = \phi^{-1}\circ \theta \circ \phi(\chi) &  = \phi^{-1}\circ \theta \big( \chi -  x\, \theta_0(\chi)\big)\\
&= \phi^{-1}\big(x\,\theta_1(\chi) + \theta_0(\chi)  -  1\, \theta_0(\chi)  +  x\, \theta_0\circ\theta_0(\chi)\big) \\
&= \phi^{-1}\big(x\,\big(\theta_1(\chi) +  \theta_0\circ\theta_0(\chi)\big)\big) \\
&= x\,\big(\theta_1(\chi) +  \theta_0\circ\theta_0(\chi)\big),\\
\end{aligned}
$$
with the last line following from the observation made leading into \eqref{eq: phi inv 1-0}, that $\phi^{-1}(x \chi) = x\chi$.  So $\theta'$ has the form asserted.  Note that $\theta_0\circ\theta_0 = (1/2) [\theta_0, \theta_0]$  is a derivation of degree $4n+2$.   We have
$$\lambda = \theta_1 + \frac{1}{2} [\theta_0, \theta_0].$$

\noindent{}(c)  As in   \eqref{eq: theta 1-0}, we   write 
\begin{equation}\label{eq: d 1-0}
d(\chi) = x\,d_1(\chi) + d_0(\chi),
\end{equation}
which defines $d_1$ and $d_0$ as derivations of $\land W.$    We now compute:
$$
\begin{aligned}
\theta\circ d(\chi) & = \theta\big( x\,d_1(\chi) + d_0(\chi) \big)\\
&= 1\cdot d_1(\chi) -x\,\big( x\,\theta_1(d_1(\chi)) + \theta_0(d_1(\chi)) \big) +  x\,\theta_1(d_0(\chi)) + \theta_0(d_0(\chi))\\
&= d_1(\chi) -x\, \theta_0\circ d_1(\chi) +  x\,\theta_1\circ d_0(\chi) + \theta_0\circ d_0(\chi)\\
\end{aligned}
$$
and
$$
\begin{aligned}
d \circ \theta(\chi) & = d \big( x\,\theta_1(\chi) + \theta_0(\chi) \big)\\
&= dx\cdot \theta_1(\chi) -x\,\big( x\,d_1(\theta_1(\chi)) + d_0(\theta_1(\chi)) \big) +  x\,d_1(\theta_0(\chi)) + d_0(\theta_0(\chi))\\
&= dx\, \theta_1(\chi) -x\, d_0\circ \theta_1(\chi) +  x\,d_1\circ \theta_0(\chi) + d_0\circ \theta_0(\chi)\\
\end{aligned}
$$
Adding these two identities, whose left-hand sides sum to zero, gives us the following two identities amongst derivations of $\land W$:
\begin{equation}\label{eq: d-theta 0-1 ids}
[d_0, \theta_1] = [d_1, \theta_0] \quad \text{and} \quad d_1 + [d_0, \theta_0] = -dx\,\theta_1
\end{equation}
with the first from collecting terms in $x\cdot\land W$ and the second from collecting terms in $\land W$, which terms are independent of each other.

Now we calculate:
$$
 \begin{aligned} 
 d'(\chi) &= \phi^{-1}\circ d \circ \phi(\chi) =   \phi^{-1}\circ d \big( \chi -  x\, \theta_0(\chi)\big)\\
&= \phi^{-1}\big(x\,d_1(\chi) + d_0(\chi)  -  dx\, \theta_0(\chi)  +  x\, d_0\circ\theta_0(\chi)\big) \\
&= x\,d_1(\chi) + \big( d_0(\chi) + x\,\theta_0\circ d_0(\chi)\big)  \\
& \hbox{\hskip1truein} - \big( dx + x\, \theta_0(dx)\big)\big( \theta_0(\chi) + x\, \theta_0\circ\theta_0(\chi) \big) + x\, d_0\circ\theta_0(\chi)\\
&= x\,\big( d_1(\chi)  + [d_0, \theta_0](\chi) - dx\, \theta_0\circ\theta_0(\chi)\big) + d_0(\chi)  - dx\, \theta_0(\chi).\\
\end{aligned}
$$
The last line follows using  $\theta_0(dx) = 0$, which follows for degree reasons, to cancel one term from the penultimate line.  Notice that we also commuted $x$ and $dx$ without changing the sign. Using the second identity of 
\eqref{eq: d-theta 0-1 ids}, we arrive at
$$
\begin{aligned}
d'(\chi) &= - x\,\big(dx\,\theta_1(\chi)  + dx\, \theta_0\circ\theta_0(\chi)\big) + d_0(\chi)  - dx\, \theta_0(\chi)\\
&= -x \,dx\,\lambda(\chi) +  d'_0(\chi),
\end{aligned}
$$
where we have written $d'_0(\chi) = d_0(\chi)  - dx\, \theta_0(\chi)$ for the term not involving $x$.  One can see this is a derivation either  as  it is a sum of two terms, each of which acts as a derivation on $\land W$.   Alternately,  the terms here may be written in the form $d'(\chi) = x d'_1(\chi) + d'_0(\chi)$, with both $d'_1$ and $d'_0$ derivations of $\land W$ by the observation at the top of the proof.  
\end{proof}
 
\begin{remarks}  \label{rem:odd} We record some further consequences.  
\begin{itemize}
\item[(1)]    If $dx = 0$ then, by (c),    there is a  DG algebra decomposition $$(\land V, d') \cong (\land(x), 0) \otimes (\land W, d_0').$$ We thus recover the odd-degree case of \cite[Lem.1]{Hal88}.      \\ 
\item[(2)] 
If we decompose   $[\theta, \theta]$    in the form 
$[\theta, \theta] = x \, [\theta, \theta]_1 + [\theta, \theta]_0,$
with $[\theta, \theta]_1$ and $[\theta, \theta]_0$ derivations of $\land W$.  Then  $\lambda$ in part (b) is given by   $$\lambda = \frac{1}{2}\,[\theta, \theta]_0.$$  
 \item[(3)] 
Decomposing  $d'$ in this form so that  
$d'(\chi) = x\,d'_1(\chi) + d'_0(\chi)$
we then have the identities:   %
$$[d'_1, d'_0] = 0 \hbox{\, and \, } d_0'\circ d'_0 + dx\,d'_1 = 0.$$
\end{itemize}
\end{remarks}

An  even-degree Gottlieb element $y \in V^{2m}$ gives rise to a basis change, as well.  In this case, the result  gives Theorem \ref{terminal}, that $y$ is a terminal element.  We explain this now.  Let $\theta$  be the   derivation cycle of degree $2m$ with 
 $\theta(y) = 1$ and  $[d, \theta] = d\theta - \theta d = 0$.   
 Write  $V = \langle y \rangle \oplus W$ and  define a linear map $\psi \colon \land V \to \land V$ by setting
$$\psi(y) = 0 \quad \text{and} \quad \psi(\chi) = - y \theta(\chi)$$
for $\chi \in \land W$.  Extend  multiplicatively so that the ideal of $\land V$ generated by $y$ is in the kernel of $\psi$.  Then $\psi$ is a degree-zero derivation of $\land V$ and, further, $\psi$ is {\em locally nilpotent}  meaning that, for each element $\xi \in \land V$, there is some $r$ for which $\psi^r(\xi) = 0$.  We may thus exponentiate $\psi$ to obtain a linear map
$$\phi = \mathrm{exp}(\psi) = \mathrm{id} + \psi + \frac{1}{2!}\, \psi^2  + \frac{1}{3!}\, \psi^3 + \cdots.$$
Then   $\phi$ is  an automorphism of $\land V$ and   so induces a  DG algebra isomorphism %
$$\phi \colon (\land V, d') \to (\land V, d)$$
 by setting $d' = \phi^{-1}\circ d \circ \phi$.  We  transfer the derivation cycle $\theta$ to one of $(\land V, d')$ by setting $\theta':= \phi^{-1}\circ \theta \circ \phi$.

\begin{theorem}\label{even}
Let $(\land V, d)$ be a Sullivan minimal model with a Gottlieb element $y \in V^{2m}$.  Let $\phi \colon (\land V, d') \to (\land V, d)$    and $\theta'$ the induced derivation cycle of $(\land V, d')$ be as above.  Write $V = \langle y \rangle \oplus W$.     Then with notation as above:  \begin{itemize}
\item[(a)]  $\theta'(y) = 1$ and $d'(y) =dy$   \\
\item[(b)] For $\chi \in \land W$, we have  $\theta'(\chi) = 0$.  \\
\item[(c)] The differential $d'$ restricts to a derivation of $\land W$, namely, we have $$d'(\land W) \subseteq \land W.$$ %
\end{itemize}
\end{theorem}

\begin{proof}
Again  (a) is immediate.  For (b),   let $\chi \in \land W$  and observe
$$\phi(\chi) = \chi + \psi(\chi) + \frac{1}{2!}\, \psi^2(\chi)  + \frac{1}{3!}\, \psi^3(\chi) + \cdots.$$
Notice that $\psi^2(\chi) = \psi\big(-y\theta(\chi)\big) = (-y)(-y) \theta^2(\chi)$, since we defined $\psi(y) = 0$.  Generally, we have
$$\psi^k(\chi) =  (-1)^{k} y^k  \theta^k(\chi),$$
for $k \geq 1$.  It follows that
$$\phi(\chi) = \chi + \sum_{k \geq 1}\ (-1)^k\, \frac{1}{k!}\,y^k \theta^k(\chi).$$
Recall that $\theta$ is locally-nilpotent, so the sum is finite for each $\chi.$  
Applying the derivation $\theta$, with $\theta(y) = 1$, we have 
$$
\begin{aligned}
\theta\big(\phi(\chi)\big) &= \theta(\chi) - \theta(\chi) - y \theta^2(\chi)  \\
& \hbox{\hskip1truein}+ \sum_{k \geq 2}\ (-1)^k\left( \, \frac{1}{(k-1)!}\,y^{k-1} \theta^k(\chi) + \frac{1}{k!}\,y^k \theta^{k+1}(\chi)\right)\\
&= \theta(\chi) - \theta(\chi) - y \theta^2(\chi) + y \theta^2(\chi) \\
&\hbox{\hskip1truein}+ \sum_{k \geq 3}\ (-1)^k\, \frac{1}{(k-1)!}\,y^{k-1} \theta^k(\chi) + \sum_{k \geq 2}\ (-1)^k\,\frac{1}{k!}\,y^k \theta^{k+1}(\chi)\\
&= \theta(\chi) - \theta(\chi) - y \theta^2(\chi) + y \theta^2(\chi) \\
&\hbox{\hskip1truein}+ \sum_{k \geq 2}\ \left((-1)^{k+1} \, \frac{1}{k!}\,y^{k} \theta^{k+1}(\chi) + (-1)^k \,\frac{1}{k!}\,y^k \theta^{k+1}(\chi)\right)\\
&= 0.
\end{aligned}
$$
Finally, this gives $\theta'(\chi) = \phi^{-1}\circ \theta' \circ \phi(\chi) = \phi^{-1}\big( x\lambda\big( \phi(\chi)\big)\big) = 0$, as claimed. 

\ \ 

(c) Let $v\not= y$ be a generator of degree at least $2m-1$ and write 
$$d'(v) = \chi_0 + y \chi_1 + \cdots + y^k \chi_k,$$
for some $k \geq 1$ and each $\chi_i \in \land W$, that is, not containing terms that involve $y$. The derivation $\theta'$ satisfies $\theta'(y) = 1$ and $\theta'(\chi_i) = 0$ for each $i$.  Also, because $\theta'$ is a cycle with respect to $D' = \ad(d')$, we have 
$$\theta'\circ d'(v) = d'\circ \theta'(v) = 0,$$
and hence we have
$$0 = 0 +  \chi_1 +  y \chi_2 + \cdots +  y^{k-1} \chi_k.$$
Each of these terms must be zero independently of each other, since the $\chi$ do not involve $y$, an even-degree generator of $\land V$.  Hence, each $\chi_i = 0$ for $i = 1, \ldots, k$, and we have $d'(v) = \chi_0$, which does not involve $y$.  For generators $v$ of degree $2m-2$ and lower, $d'(v)$ cannot involve $y$ for degree reasons.    \end{proof}

Observe that the even-degree case of   \cite[Lem.1]{Hal88} follows from Theorem \ref{even}.    When  $dy = 0$, we have a splitting $(\land V, d) \cong (\land(y), 0) \otimes (\land W, d').$ 
\begin{proof}[Proof of Theorem \ref{terminal}]
 Observe that Theorem \ref{terminal} is a direct consequence of Theorem \ref{even}. We give an alternate,    homotopy-theoretic  proof  of Theorem \ref{terminal}  in Section \ref{sec5}.  
 \end{proof}
\section{Location of Rational Gottlieb Elements}
\label{sec3}
We begin  with  a simple  example of a non-terminal Gottlieb element, necessarily  of odd degree.  
\begin{example} \label{ex}
Consider the minimal model $(\land V, d)$ with $V = \langle u_2, v_2, x_3, z_3 , y_6 \rangle$, where subscripts denote degrees, and differential given on generators by
$$du = 0 = dv, dx = uv, dz = v^2, dy = -uv x + u^2 z.$$
Set $\theta(x) = 1$ and $\theta(y) = x$, and $\theta = 0$ on all other generators.  It is straightforward to check that, extended as a derivation of $\land V$, we have $[d, \theta] = 0$.   
\end{example}

 Observe that the even-degree element $y \in V^6$ is also a Gottlieb element in Example \ref{ex} as it is a terminal element.   Indeed if  $V$ is finite-dimensional then  we see that any $v \in V^N$ with $N = \mathrm{max}\{n \mid V^n \neq 0\}$ is  terminal and so a  Gottlieb element.    The converse to this observation is relevant to the realization problem for classifying spaces, in particular,  Theorem \ref{introthm: S4} above.  Namely, we note that   a space  having a non-vanishing rational homotopy element in a  degree that is higher  than the degree of all non-vanishing Gottlieb elements must be a $\pi$-infinite space.    
  
 Suppose now that     $(\land V, d)$ is  a  formal,  minimal DG algebra.  When  $V$ is finite-dimensional then, in the lower grading, we have
$V = V_0 \oplus V_1$ by  \cite[Th.2]{FH82}.   Conjecture \ref{conj: Jacobsson}  concerns the case when $V$ is infinite-dimensional.  Recall the assertion is     that a Gottlieb element  $v \in V^n$ should be of lower grading $\leq 1$, i.e.,  $v \in V_0 \oplus V_1$ with $V_0 = Z(V)$  the space of cycles.   The following argument is due to Dupont \cite[Pro.4]{Dupont}. 

\begin{proof}[Proof of Theorem \ref{J}]  
Let $X$ be  a formal space with Sullivan minimal model $(\land V, d)$   and suppose $H^*(X; \Q) \cong H(\land V, d)$ is  finitely generated.  Then $V_0$ is finite-dimensional.   Suppose    $y \in V^{2n}$  is a Gottlieb element.   We    prove that $dy = 0$ so that $y \in V_0$.  

Use the Gottlieb element $y$ and  Theorem \ref{even} to obtain a change of basis   $(\land V, d')$ so that, in $(\land V, d'),$ the generator $y$ does not appear in the differential $d'$ and the derivation $\theta'$ has $\theta'(y) = 1$ and $\theta' = 0$ on all other generators.  Further, we have $d'(y) = dy.$ 

Suppose $d'(y) \neq 0$ so that $y \not\in V_0.$   Write $V_0^{\mathrm{even}} = \langle z_1, \ldots, z_k \rangle.$   
Let $a_i$ be   of degree $2|z_i| -1$.   Extend  $(\land V, d')$ to a minimal DG algebra    $(\land V \otimes \land(a_1, \ldots, a_k), \delta)$ by setting  $\delta(a_i) = z^2_i$ and $\delta = d'$ on $\land V.$   We prove:  
\begin{lemma} \label{Hfinite} $\mathrm{dim}\,  H(\land V \otimes \land(a_1, \ldots, a_k), \delta) < \infty.$
\end{lemma} 
\begin{proof}   Write $\mathcal{H} =  H(\land V, d')$ and 
$\mathcal{H}_0 = \mathcal{H} / \langle x_1^2, \ldots, x_k^2 \rangle.$   Since $(\land V, d')$  is formal we have  a quasi-isomorphism $(\land V, d') \simeq (\mathcal{H}, 0)$ and so a corresponding quasi-isomorphism 
$$(\land V \otimes \land(a_1, \ldots, a_k), \delta) \simeq (\mathcal{H}  \otimes \land(a_1, \ldots, a_k), \delta')$$
with $\delta'(a_i) = x_i^2$ (where $x_i$ now denotes the cohomology class represented by $x_i$) and $\delta' = 0$ on $\mathcal{H}.$  
Observe that   the graded algebra $\mathcal{H}  \otimes \land(a_1, \ldots, a_k)$ is  finitely generated as a module over $\mathcal{H}$.  Since $\mathrm{ker} \,  \delta'$ is preserved by the action of $\mathcal{H}$, we see  that the homology
$H(\mathcal{H}  \otimes \land(a_1, \ldots, a_k), \delta')$ is   a module over $\mathcal{H}$ as well, and finitely generated as such.   In the latter $\mathcal{H}$-module, the elements $x_i^2 \in \mathcal{H}$ act trivially.     For if $\gamma \in \mathcal{H}  \otimes \land(a_1, \ldots, a_k)$  is a $\delta'$-cycle  then $x_i^2 \gamma = \delta'(a_i \gamma)$ is a $\delta'$-boundary.    We conclude that $H(\mathcal{H}  \otimes \land(a_1, \ldots, a_k), \delta')$ 
is a  module over $\mathcal{H}_0$  and is finitely generated as such.  But $\mathcal{H}_0$ is   clearly finite-dimensional. Thus    
 $ H(\land V \otimes \land(a_1, \ldots, a_k), \delta)  \cong H(\mathcal{H}  \otimes \land(a_1, \ldots, a_k), \delta')$ is a finitely-generated module over a finite-dimensional algebra $\mathcal{H}_0$ and is thus finite-dimensional.  
 \end{proof} 
%The   sequence   $$(\land V, d') \to (\land V \otimes \land(a_1, \ldots, a_k), \delta) \to (\land(a_1, \ldots,  a_k), 0) $$  of DG algebras corresponds to a rational fibration $F_\Q \to E_\Q \to X_\Q$.   The fibre $F_\Q$ is a finite product of odd-dimensional spheres.  Since $d'y \neq 0$, a standard Serre spectral sequence argument implies that the algebra $H^*(\land V \otimes \land(a_1, \ldots, a_k), \delta)$ is finite-dimensional.  Therefore $E_\Q$ has finite    L.S.  category by \cite[Cor.2, p.386]{FHT}. 

 Now consider  the derivation $\theta'$ of $(\land V, d')$.  Since $y$ does not appear in any differential $d'(v)$ for $v \in V'$, we may extend $\theta'$ to a derivation $\theta''$ of  $(\land V \otimes \land(a_1, \ldots, a_k), \delta)$, that satisfies $[\delta, \theta''] = 0$ and $\theta''(y) = 1$ by   setting  $\theta''(a_i) = 0$ for each $i$.   Now $\theta''$ displays $y$ as an even-degree Gottlieb element in the Sullivan minimal model  $(\land V \otimes \land(a_1, \ldots, a_k), \delta)$.  By Lemma \ref{Hfinite},  $(\land V \otimes \land(a_1, \ldots, a_k), \delta)$ has  finite rational L.S. category contradicting   \cite[Pro.29.8(ii)]{FHT}. We conclude that $dy = d'(y)= 0.$      
\end{proof}

 \begin{remark}   We note the need for Theorem \ref{terminal} in the preceding.  The extension of the derivation $\theta'$ to $\theta''$ requires $y \in V^{2n}$ to be a terminal element. Otherwise it is not clear how the derivation $\theta''$ is to be defined on the $a_i$.      \end{remark}

We  next apply our basis change formula to affirm  Conjecture \ref{2Nconjecture}  in a restricted case.    
 Recall our hypothesis in Theorem \ref{2N} is   that there is  an element $x \in V^{2n+1}$ that is a Gottlieb element and such that $dx$ a monomial. We introduce  notation for the latter hypothesis. 
Write $V = \langle u_1, \ldots, u_k \rangle$ in a basis with the $u_i$   in (non-strictly) increasing  order of degrees.  Then {\em $dx$ is a monomial} means we may  write 
$$dx = w_1 \cdots w_n$$
in which $w_j = u_{i_j}^{m_j}$ for $m_j \geq 1$.  We assume the indices $i_j$ are increasing and so the degrees of the $u_{i_j}$ are increasing.  Note  $m_j = 1$ when  $u_{i_j}$ has odd degree.  \begin{lemma} \label{w1} With notation as above, we have $dw_1 = 0.$ 
\end{lemma}
\begin{proof}
We have $$ 0 = d^2x = d(w_1\cdots w_n) =d(w_1\cdots w_{n-1}) (w_n) + (-1)^{|w_1| + \cdots + |w_{n-1}|}(w_1\cdots w_{n-1})dw_n .$$
Since $u_{i_n}$ is of maximal degree,   the  term  $u_{i_n}^{m_n}$ does not  appear in $(w_1\cdots w_{n-1})dw_n$.  We must have
that the two summands above are both zero.  In particular, $$d(w_1\cdots w_{n-1}) (w_n) = 0.$$   Since  $u_{i_n}$ cannot appear in  $d(w_1\cdots w_{n-1})$ we have, in fact,  $d(w_1\cdots w_{n-1}) = 0.$  The result now follows by induction. 
\end{proof}

\begin{proof}[Proof of Theorem \ref{2N}] 
 Suppose we are given a Sullivan minimal model $(\land V, d)$ with $H^q(\land V, d) = 0$ for $q > N$   and $x \in V^{2n+1}$
  a Gottlieb element such that $dx$ is a monomial.  We must prove $n \leq N-1.$  We perform  the basis   change  of Theorem \ref{odd} using $x \in V^{2n+1}$ to obtain an isomorphic Sullivan minimal model $(\land V, d')$.   We observe that $d'(x) = dx$ remains a monomial under this basis change.    
    
 Suppose now that  $n \geq N.$   Write $d'(x) = w_1\ldots w_n$ as above so that $d'(w_1) = 0$ by Lemma \ref{w1}.  
  Suppose first that $w_1$  is of odd degree.  Let $u = w_1x$ and observe that 
  $$d'(u) = d'(w_1x) = (-1)^{|w_1|}w^2_1 w_2\cdots w_n = 0$$
  since $w_1^2 = 0.$ Now $|u| > N$ and so $u$ must be a boundary of $(\land V, d').$
  However, this contradicts Theorem \ref{odd} (3).  If $x$ appears in a differential  $d'(\chi)$ for $\chi \in \land V$ then it must appear in a  term of the form $x\lambda(\chi)d'(x).$

  So suppose  $w_1$ is  of even degree.  In this case,  the identity
  $$0 = (d')^2x = d'(w_1 \cdots w_n) = (-1)^{|w_i|} w_1 d'(w_2\cdots w_n)$$
  implies  $d'(w_2 \cdots w_n)=0$.   Now, since $|d'(x)| = 2n+2>2N$,
  either $|w_1|> N$ or $|w_2\cdots w_n| > N.$  In the first case, $w_1$ is a cycle of $(\land V, d')$ of degree greater than  $N$  and so $w_1$ is a boundary.   If  $w_1 = d' (\eta)$ for $\eta \in \land V$ 
 then $u = x - \eta w_2 \cdots w_n$ is a cycle with $|u| > N$.    Note that  $u$ cannot be a boundary  
 as this contradicts the decomposability  of $d'.$  Similarly, if $w_2 \cdots w_n$ has degree $> N$ then
 it must bound, say $d'(\zeta) =   w_2 \cdots w_n.$  Again this gives a cycle $u = x  - (-1)^{|w_1|}w_1\zeta$ 
 with $|u| > N$ that cannot be exact.
 
 \end{proof}

 \section{ The realization problem for $S^{2n}$ }\label{sec4}
 The problem of  realizing  $S^{2n}$  as a classifying space,    of producing a   rational space $X_\Q$ with  $\B(X_\Q) \simeq   S_\Q^{2n}$,      translates, with Sullivan's identity (\ref{eqS}),     to a simple algebraic problem.   Namely, a solution to the realization problem for $S^{2n}$    is  a Sullivan minimal model $(\land V, d)$    giving a  Lie algebra isomorphism:   $$H_*(\Der(\land V), D) \cong  \pi_*(\Omega S^{2n}_\Q) \cong   \langle \iota_{2n-1}, [\iota_{2n-1}, \iota_{2n-1}] \rangle.$$ Here $\iota_{2n-1} \in \pi_{2n-1}(\Omega S^{2n}_\Q) \cong \pi_{2n-1}(S^{2n-1}_\Q)$ corresponds to the fundamental class.    
  
Throughout this section, we assume a solution to the realization problem is given. Specifically, we assume we have  a Sullivan minimal model  $(\land V, d)$  admitting   a non-bounding  derivation cycle $\theta_a \in \Der_{2n-1}(\land V)$     such that   $ [\theta_a, \theta_a]$ is non-bounding, as well,  and that these two classes  span the homology of derivations:   \begin{equation} \label{ab}  H_*(\Der (\land V), D) = \langle \theta_a, [\theta_a, \theta_a]  \rangle. \end{equation}    
Write $\theta_b = [\theta_a, \theta_a]$.
 We  make one  further hypothesis.    We assume     \begin{equation} \label{yGottlieb} \theta_b(y_0) = 1 \hbox{\,  for some \, } y_0 \in V^{4n-2}. \end{equation}  Note that the   hypothesis (\ref{yGottlieb}) corresponds to the case that $(\land V, d)$ has a nontrivial Gottlieb element   in degree $4n-2$.     
 
 We prove that assumptions (\ref{ab}) and (\ref{yGottlieb})  lead  to a contradiction when $n = 1,  2,$ and $3$.  For the latter two cases, we will make use of a combined version of our change of basis results  from Section \ref{sec2}.  
We begin with the following:  
\begin{lemma} \label{Gottliebx}   
There is $x \in V^{2n-1}$ with $\theta_a(x) = 1.$ 
\end{lemma}
\begin{proof}
By (\ref{yGottlieb}) we have  $$ 1 =   \theta_b(y_0) =  [\theta_a, \theta_a](y_0)  = 2\theta_a(\theta_a(y_0)) = \theta_a\left(2\theta_a(y_0)\right)  .$$
 It follows that $2\theta_a(y_0)$ has an indecomposable summand, i.e.,   $2\theta_a(y_0) = x + \chi$ for some  $x  \in V^{2n-1}$ and $\chi$ decomposable in $\land V.$  Then $\theta_a(\chi)$ is decomposable and so, by the above equation, $\theta_a(\chi) =0.$ Thus  $\theta_a(x) = 1$,  as needed.     
\end{proof}  
Lemma \ref{Gottliebx}  already leads to a contradiction in the case  $n = 1$ since   $V^1 = 0$.  We   note a stronger statement can be made in this case (see Proposition \ref{S2}, below).      We assume now  that $n \geq 2$ 
\begin{lemma} \label{dx}  The element $x \in V^{2n-1}$ with $\theta_a(x) = 1$ satisfies $dx \neq 0.$  
\end{lemma}  
\begin{proof}
If $dx = 0$   then     by Remark \ref{rem:odd} (1) we  have  a factorization of DG algebras:  $$(\land V, d) \cong (\land(x), 0) \otimes (\land W, d).$$   By the above factorization, we see   $(x, 1)$ is a non-bounding derivation cycle of degree $2n-1$. However, $[(x, 1), (x, 1)] = 0$ which is a contradiction of   (\ref{ab}).     \end{proof}

 We now apply our work in  Section \ref{sec2}.   First we   make the  change of basis  $\phi \colon (\land(W, x), d') \to (\land V, d)$ in 
\eqref{eq: odd change of basis}  using the Gottlieb element   $x \in V^{2n-1}$.  In the notation of Theorem \ref{odd},   we see that
$\phi(y_0) = y_0$  since $y_0 \in W$.  Let $y = 2y_0.$  Remark \ref{rem:odd} (2) implies $$\lambda(y) = \frac{1}{2}[\theta_a, \theta_a]_0(2y_0) =   \theta_b(y_0) =1.$$
By Theorem \ref{odd} (b), we have  $$\theta'(a)(y)  = \lambda(y)x = x.$$

We next apply a basis change as in Theorem \ref{even} except in this case we use the derivation $\lambda \in \Der_{4n-2}(\land V)$.  
Write  $V = \langle x\rangle \oplus  \langle y\rangle \oplus \widehat{V}$, so that $W = \langle y\rangle \oplus \widehat{V}$ in the notation of Theorem \ref{odd}.  First define a linear map $\psi \colon \land V \to \land V$ by    setting
$$\psi(x) = 0, \quad \psi(y) = 0, \quad \text{and} \quad \psi(\chi) = - y \lambda(\chi)$$
for $\chi \in \land \widehat{V}$, and then extending $\psi$ multiplicatively (so that the ideal of $\land V$ generated by $x$ and $y$ is in the kernel of $\psi$).  Then $\psi$ is a degree-zero, locally nilpotent derivation of $\land V$.  We may then  exponentiate $\psi$ to an isomorphsim %
$$\phi' = \mathrm{exp}(\psi) = \mathrm{id} + \psi + \frac{1}{2!}\, \psi^2  + \frac{1}{3!}\, \psi^3 + \cdots \colon \land V \to \land V$$
 Setting 
$d'' = \phi'^{-1}\circ d' \circ \phi',$ we  obtain a DG algebra isomorphism 
$$\phi' \colon (\land V, d'') \to (\land V, d')$$
We transfer the derivation cycle $\theta_a'$ to this DG algebra by  $\theta_a'' = \phi'^{-1}\circ \theta_a' \circ \phi'$.  Then $[d'', \theta_a''] = 0.$  We prove:
\begin{proposition}\label{prop: odd even gottlieb structure}
 Let  $(\land V, d'')$ and   $\theta_a'' \in \Der_{2n-1}(\land V)$ be as constructed above.    Write $V = \langle x\rangle \oplus W$ where $W =   \langle y\rangle \oplus \widehat{V}$.  Then    \begin{itemize}
\item[(a)]  $d''x = d'x = dx$ and   $\theta_a''(x) = 1, \ \   \theta_a''(y) = x, \hbox{\, and, \,} \theta_a''(\chi) = 0$ for $\chi \in \land \widehat{V}$.  \\
\item[(b)]   $y$ does not appear in the differential $d''(w)$ for any $w \in V$    \\
\item[(c)]    Given $\chi \in \land W$ and decompose $d''$ in the form  $d''(\chi) = xd_1''(\chi) + d_0''(\chi)$
 for derivations $d''_1$ and $d''_0$ of $\land W $.  We then have 
  $$d''_1(y) = - dx \hbox{\, and \,} 
d_1''(\chi)  = 0 \hbox{\,  for $\chi \in \land \widehat{V} $}$$ 
 (namely, $x$ does not appear in any differential other than that of $y$, where it occurs only in  the term $-x dx$).
\end{itemize}
\end{proposition}

\begin{proof}
(a) Note  $d'(x) = dx$ by Theorem \ref{odd} (a)  while $d''(x) = d'(x)$ since  $\psi(x) = 0$. Since  $\psi(y) = 0$,  it is immediate that $\phi'(x) = \phi'^{-1}(x) = x$ and $\phi'(y) = \phi'^{-1}(y) = y$.  It then follows that we have $\theta_a''(x) = 1$
and  $\theta''(y) = x$.  For $\chi \in \land \widehat{V}$, we have 
$$\phi'(\chi) = \chi + \psi(\chi) + \frac{1}{2!}\, \psi^2(\chi)  + \frac{1}{3!}\, \psi^3(\chi) + \cdots.$$
But notice that $\psi^2(\chi) = \psi\big(-y\lambda(\chi)\big) = (-y)(-y) \lambda^2(\chi)$, since  $\psi(y) = 0$.  Generally, we have
$$\psi^k(\chi) =  (-1)^{k} y^k  \lambda^k(\chi),$$
for $k \geq 1$.  It follows that
$$\phi'(\chi) = \chi + \sum_{k \geq 1}\ (-1)^k\, \frac{1}{k!}\,y^k \lambda^k(\chi).$$
Then, applying the derivation $\lambda$, with $\lambda(y) = 1$, we have 
$$
\begin{aligned}
\lambda\big(\phi'(\chi)\big) &= \lambda(\chi) - \lambda(\chi) - y \lambda^2(\chi)  \\
& \hbox{\hskip1truein}+ \sum_{k \geq 2}\ (-1)^k\left( \, \frac{1}{(k-1)!}\,y^{k-1} \lambda^k(\chi) + \frac{1}{k!}\,y^k \lambda^{k+1}(\chi)\right)\\
&= \lambda(\chi) - \lambda(\chi) - y \lambda^2(\chi) + y \lambda^2(\chi) \\
&\hbox{\hskip1truein}+ \sum_{k \geq 3}\ (-1)^k\, \frac{1}{(k-1)!}\,y^{k-1} \lambda^k(\chi) + \sum_{k \geq 2}\ (-1)^k\,\frac{1}{k!}\,y^k \lambda^{k+1}(\chi)\\
&= \lambda(\chi) - \lambda(\chi) - y \lambda^2(\chi) + y \lambda^2(\chi) \\
&\hbox{\hskip1truein}+ \sum_{k \geq 2}\ \left((-1)^{k+1} \, \frac{1}{k!}\,y^{k} \lambda^{k+1}(\chi) + (-1)^k \,\frac{1}{k!}\,y^k \lambda^{k+1}(\chi)\right)\\
&= 0.
\end{aligned}
$$
Finally, this gives $\theta_a''(\chi) = \phi'^{-1}\circ \theta_a' \circ \phi'(\chi) = \phi'^{-1}\big( x\lambda\big( \phi'(\chi)\big)\big) = 0$, as claimed.

(b) Let $v$ be a generator of degree at least $4n-3$ and write 
$$d''(v) = \chi_0 + y \chi_1 + \cdots + y^k \chi_k,$$
for some $k \geq 1$ and each $\chi_i \in \land(x) \otimes \land(\widehat{V})$, that is, not containing terms that involve $y$. The derivation $[\theta_a'', \theta_a''] = 2 \theta_a''\circ\theta_a''$ satisfies $[\theta_a'', \theta_a''](y) = 2$.    Now $[\theta_a'', \theta_a''](v) = 0$ for  $v \in \widehat{V}$ by (a).  Also, because $[\theta_a'', \theta_a'']$ is a cycle with respect to $D'' = \ad(d'')$, we have 
$$[\theta_a'', \theta_a'']\circ d''(v) = d''\circ [\theta_a'', \theta_a''] (v) = 0,$$
and hence we have
$$0 = 0 + 2\lambda \chi_1 + 4 \lambda y \chi_2 + \cdots + 2k \lambda y^{k-1} \chi_k.$$
Here, we use the fact that $[\theta_a'', \theta_a''] (\chi_i) = 0$ for each $i$.   But each of these terms must be zero independently of each other, since the $\chi$ do not involve $y$, an even-degree generator of $\land V$.  Hence, each $\chi_i = 0$ for $i = 1, \ldots, k$, and we have $d''(v) = \chi_0$, which does not involve $y$.  For generators $v$ of degree $4n-4$ and lower, $d''(v)$ cannot involve $y$ for degree reasons.  So $y$ does not appear in $d''$, as claimed.  

(c) Suppose given $v \in \widehat{V}^k$  with $k > 2n-2$.  Write 
$$d''(v) = x V_1 + V_2,$$
for $V_1, V_2 \in \land \widehat{V}$---recall that we just showed in (b) that  $y$ does not occur in the differential of any generator.    Then we have $\theta_a''(dv) = V_1$, since $\theta_a''(\widehat{V}) = 0$.   On the other hand, we have $\theta_a''(v) = 0$, for the same reason, and because $\theta_a''$ is a $D''$-cycle  we have 
$$0 = d''\circ\theta_a''(v) = - \theta_a''\circ d''(v) = - V_1.$$
Thus   $d''(v) = V_2 \in \land\widehat{V}$,  and $x$ does not appear in $d''(v)$.    For generators $v$ of degree $2n-2$ and lower, $d''(v)$ cannot involve $x$ for degree reasons.
Finally, from part (b) of Theorem \ref{odd}, we have $d'(y) = -x\,dx + d'_0(y)$, and $d'_0(y)$ is a decomposable term, since the original $d$ was decomposable. It follows that  $\phi'^{-1}\big(-x\,dx + d'_0(y)\big) = -x\,dx + d'_0(y)$, and hence $d''(y) = -x\,dx + d''_0(y)$.   
\end{proof}

We apply   Proposition \ref{prop: odd even gottlieb structure},    to  obtain a differential $d''$ on $\land V$.  Then   $y \in V^{4n-2}$ does not appear in a differential $d''(w)$ for any $w \in V$, i.e. $y$ is terminal in $(\land V, d'')$. We may write \begin{equation} \label{d''} d''(y) = - x\,d''(x) + d''_0(y),\end{equation}
(here using that $d''(x) = dx$) where  $d''_0(y) \in \land W$ (so not involving $x$).  Furthermore, $x$ does not appear in the differential $d''(v)$ for $v$ in the  subspace $\widehat{V}$  complementary   to $y$ in $V$.

\begin{proposition}  \label{dim}  Let  $n =2 $ or $3$.   Let  $(\land V, d'')$ be the  Sullivan minimal model described above.  Then   $$\dim \, H_*(\Der(\land V), D)> 2.$$ 
 \end{proposition}
 \begin{proof}
 For convenience in the proof,  we omit double subscripts and write $d = d''$.     We begin with the case $n = 2$. By Lemma \ref{dx}, we may suppose that $dx \not = 0$.     Write $V^2 = \langle v_1, ..., v_r \rangle$ and note that, for degree reasons, $V^2$ consists of cycles: $d(v_j) = 0$ for each $j = 1, \ldots, r$.  Now $dx \in \land^2 V^2$, so choose and fix elements $\alpha_j \in V^2$ for which
$$dx = v_1 \alpha_1 + \cdots + v_r \alpha_r$$
(some, but not all, of the $\alpha_j$ may be zero).    For each $v_j \in V^2$, the derivation $(y, v_j)$ is a $D$-cycle since $y$ does not appear in the differential $d$.   As a cycle of degree $4$, each $(y, v_j)$ must be exact or the inequality is achieved.  We record a consequence of this exactness in the following  form.   
  \begin{lemma} \label{exact} 
Let  $z \in V^t$ and $v\in V^s$ with $t > s$  and suppose given a derivation  of the form  $\theta = (z, v) + \theta'$ with $\theta'(z) = 0.$ If $\theta$ is a $D$-boundary, $\theta = D(\eta),$   then there exists $v^* \in V^{t+1-s}$ such that  $$dy = -x dx +  v v^* + \alpha,$$
for $\alpha \in \land \overline{V}$ where $V = \langle x, v^* \rangle \oplus \overline{V}.$
 The  derivation $\eta$ satisfies   $\eta(v^*) = 1.$
\end{lemma} 
\begin{proof}  
  Consider the expression   $v = D(\eta)(y)  =  d \eta(y) + (-1)^{|\eta|}\eta(dy)$.   Since $d$ is decomposable,  so is $d\eta(y)$ and so this term vanishes.  Thus $v = (-1)^{|\eta|} \eta(dy)$.  Thus $dy$ must contain a quadratic term $vv^*$ with $v^* \in V^{t-s+1}$ and with  $\eta(v^*) = \pm 1.$  
 The form given for the differential $dy$ follows from Proposition \ref{prop: odd even gottlieb structure}.   \end{proof} 

Returning to our case,   each degree $4$ derivation cycle $(y, v_j)$  must be exact, say $(y, v_j) = D(\eta_j)$.  By Lemma \ref{exact},   there exist  (independent)  $v_j^* \in V^5$  such that 
$$dy = -x dx + \sum_{j=1}^r\ v_j v_j^* + \beta,$$
 and  $\eta_j(v_j^*) = 1$.  Here $\beta$ is in $\land^7(\overline{V})$ with $\overline{V}$   complementary to $\langle v_1^*, \ldots, v_r^* \rangle$ in $V$.  

Now define a degree $3$ derivation
$$\gamma =  (y, x) - \alpha_1 \cdot \eta_1 - \cdots - \alpha_r \cdot \eta_r$$
Notice that, since $d(\alpha_j)$, we have $$D(\alpha_j \cdot  \eta_j) = \alpha_j D(\eta_j) = \alpha_j (y, v_j) = (y, v_j \alpha_j).$$  Also, since $y$ does not appear in the differential $d$, we have $D(y, x)  = (y, dx)$.  Hence,  %
$$D(\gamma) = (y, dx) - (y, v_1 \alpha_1) - \cdots - (y, v_r \alpha_r) = 0.$$
Thus, $\gamma$ is a $D$-cycle.  Since  $\theta_a(x) = 1$, $\gamma$ is not a multiple of $\theta_a$.  Thus $\gamma$   must be exact.  Applying Lemma \ref{exact} with $z= y, v =x,$ we   conclude that $dy$ contains a term of the form $x x^*$ for an indecomposable $x^*$ of degree $3$ with
$\eta(x^*) = 1$.  We have arrived at a contradiction: other than in the term $x dx$, the generator $x$ does not occur in $dy$.

For the case $n = 3$, we  have  $x \in V^5$ and $y \in V^{10}.$  By Lemma \ref{dx}, we   again assume $dx \neq 0.$ 
If  $V^2 = 0,$    then $dx \in \land^{2} V^3$ and we may define the derivation $\gamma$   and proceed to a contradiction exactly as above.

Suppose then that there is some    $v \in V^2$.   We define a degree $3$ derivation $$ \zeta = (x, v) + (y,vx).  $$ We claim $\zeta$ is a $D$-cycle.      Since $dv = 0$, and since  $y$ is terminal, we see $D(y, vx) = (y, vdx)$.  Since $x$ only appears in the differential $dy$ and there as the term $-xdx$, we see $D(x, v)   = (x, v) \circ d =  (y, -vdx).$    Thus $D(\zeta) =0$ and so, for degree reasons, $\zeta$ is a $D$-boundary. Write    $\zeta = D(\rho)$ for $\rho \in \
\Der^4(\land V)$.  By  Lemma \ref{exact},       $dx$ contains a term  $vv^*$ for $v^* \in V^4$ with $ \rho (v^*) = 1$.  Further, we have  that $D(\rho)(y) = vx$ with $D(\rho)$ vanishing on a complementary subspace of $\langle x, y \rangle$ in $V.$

If  $V^2 = \langle v \rangle$ is one-dimensional,   write $d(v^*) =  vw$ for $w \in V^3$.  Note that  $dw = 0$. 
The $D$-cycle $(y, v)$ must be  a boundary,  so write $(y, v) = D(\eta)$ for $\eta$ of degree $9$.   Then,  set   $$\beta =  (y, v^*) -  w \cdot \eta.$$
 We see that  $\beta$ is $D$-cycle of degree $6$  that must be  a $D$-boundary.  By Lemma \ref{exact} again,     this implies  a term $v^* v^{**}$  appears  in $dy$ with  $v^{**} \in V^7.$ But  $\rho(v^*) = 1$ and so  	$D(\rho)(y) = \rho(dy)$ has a summand  $v^{**}$   which contradicts   $D(\rho)(y) =  vx$.    
 
It remains to handle the case $V^2 = \langle v_1, v_2, \ldots, v_n\rangle$ for $n \geq 2.$  Here we   begin by defining  $D$-cycles  
 $\zeta_j = (y,v_jx)+ (x, v_j)$ of degree $3$ as above.  These must each   bound which implies we have  derivations $\rho_j$  of degree  4 with $D (\rho_j) = \zeta_j$.  Using Lemma \ref{exact}, we    may write  
 $$dx =  v_1v_1^* + \cdots + v_nv_n^* + \chi$$ for   $v_j^*  \in V^4$ where $\chi \in \land \overline{V}$ and $\overline{V}$ is complementary to $\langle v_1^*, \cdots, v_n^* \rangle$ in $V$.   We  have   $\rho_j(v_j^*) = 1.$ 
 Consider the derivation  $$\alpha = v_1\cdot \rho_2 - v_2 \cdot \rho_1$$ of degree $2.$ We see    
    $$D(\alpha)(x) = v_2 v_1 - v_1v_2 = 0 \hbox{\,  and \, } D(\alpha)(y) = v_1v_2x - v_2v_1x = 0.$$ 
  Since $D(\rho_1) = \zeta_1$ and $D(\rho_2) =\zeta_2$ and the latter vanish on a complementary subspace to $\langle x,y \rangle$ in $V$ so does $D(\alpha).$
Thus    $\alpha$ is a $D$-cycle which,  for degree reasons,  must be  a $D$-boundary.  Write $\alpha = D(\sigma)$ for $\sigma \in \Der_3(\land V)$.  Then  $D(\sigma)(v_1^*) = \alpha(v_1^*) = -v_2$ and $D(\sigma)(v_2^*) = \alpha(v_2^*) = v_1$.  From this we deduce   there is $w \in V^3$ with $\sigma(w) = 1$ and 
 the term $-wv_2$ appears in $d(v_1^*)$ while the term  $wv_1$ appears in $d(v_2^*).$   
 
Finally,  consider the $D$-cycles of degree $8$ given by $(y, v_j)$ for $ j = 1, \ldots, r.$  Since these must be $D$-boundaries we obtain degree $9$ derivations $\eta_j$   satisfying $D(\eta_j) = (y, v_j)$.  Since  $w \in V^3$  we may write 
$dw = \alpha_1v_1 + \cdots + \alpha_rv_r$ for some elements $\alpha_j \in V^2$.   Define a degree $7$ derivation 
  $$\gamma=   (y, w) - \sum_{j=1}^{r} \alpha_j \eta_j.$$  We directly check that  $\gamma$ is a $D$-cycle and so a $D$-boundary.    Then, by Lemma \ref{exact} applied with $z = y$ and $v = w$,      we must have a generator $w^* \in V^8$ such that $dy$ contains the term $ww^*.$  Since $\sigma(w) = 1$ we see that  $w^*$ occurs in $D(\sigma)(y)$ as an indecomposable summand. 
  On the other hand, we have    $$D(\sigma)(y) = \alpha(y)  = v_1\cdot \rho_2(y) - v_2 \cdot \rho_1(y)$$  is  decomposable since $\rho_j(y)$ is of degree $6$.  This contradiction completes the proof.  
 \end{proof}

\begin{proof}[Proof of Theorem \ref{introthm: S4}]  Suppose given a simply connected space $X$ satisfying $$\B(X_\Q) \simeq  S_\Q^{2n}$$  for $n = 1,2, 3.$ The Sullivan minimal model $(\land V, d)$ for $X$ then satisfies $(\ref{ab})$ and so we have  $\theta_a \in \Der_{2n-1}(\land V)$ and   $\theta_b = [\theta_a, \theta_a] \in \
\Der_{4n-2}(\land V)$ which are non-bounding derivation cycles.   If  $X$ has a non-vanishing rational Gottlieb element in degree $> 2n-1$ then   there is a derivation cycle  $\theta \in \Der_q(\land V)$  with $\theta(v) = 1$ for some $v \in V^q$.     The minimality condition for $d$ implies  $\theta$ cannot be a $D$-boundary and so, since  $q > 2n-1,$   we must have $ q = 4n-2$ and $\theta = c\, \theta_b$ for some $c \neq 0$.  Thus, taking $y_0 = c v \in V^{4n-2}$ we have $\theta_b(y_0) = 1.$  When $ n =1$, we have the contradiction $V^1 = 0$.    Proposition \ref{dim}   gives the result for $n = 2$ and $3.$    
\end{proof}

\section{The Universal Fibration and the Realization Problem}\label{sec5}
Question \ref{ques: Baut} naturally leads to  consideration of the universal fibration.   For observe that  $Y_\Q \simeq \B(X_\Q)$  for some space $X_\Q$ implies the existence of a   fibration  
$$X_\Q \to E_{X_\Q} \to \B(X_\Q) \simeq Y_\Q$$
that is universal for fibrations with fibre $X_\Q$.  We consider this expanded view here.  Write  $\partial_\infty \colon \Omega \B(X_\Q) \to X_\Q$ for the connecting homomorphism.       We recall   $\alpha \in \pi_n(X_\Q)$ is a Gottlieb element if and only if $\alpha = (\partial_\infty)_\sharp(\beta)$ for some $\beta \in \pi_{n+1}(\B(X_\Q))$ \cite[Th.2.6]{Got}.

Suppose  now  we have such universal  fibration for $Y = S^2$.  That is, suppose we have a simply connected space $X$ with  $\B(X_\Q) \simeq S^2_\Q$.      By  Theorem \ref{introthm: S4},  $X$ has  vanishing rational Gottlieb elements.  By Gottlieb's result,  mentioned above, $p_{X_\Q} \colon E_{X_\Q} \to \B(X_\Q)$ induces a surjection on rational homotopy groups. We can also see this directly    from the long-exact homotopy sequence of the universal fibration. 
As a consequence, we deduce:        

\begin{proposition} \label{S2}  Suppose $X$ is a simply connected and satisfies $\B(X_\Q) \simeq  S_\Q^2$.  Then the universal fibration    $X_\Q \to  E_{X_\Q} \to \B(X_\Q)$ with fibre $X_\Q$ has a section.    \end{proposition}
 \begin{proof}  
 We obtain a  section  of $p_{X_\Q} \colon E_{X_\Q} \to \B(X_\Q)$ as the lift of  the identity homotopy class  $S_\Q^2 \simeq \B(X_\Q) \to \B(X_\Q)$  to $E_{X_\Q}$.   
\end{proof}
 Proposition \ref{S2} gives a strong restriction on a rational space $X_\Q$ satisfying  $\B(X_\Q) \simeq S^2_\Q$.  Namely,   $X_\Q$ must have the property that every rational fibration $X_\Q \to E_\Q \to B_\Q$ has a section.     We give an example to show   that it is possible to have a non-trivial, sectioned  fibration with base $S^2$ and  fibre $X$  such that $X$ has no rational Gottlieb elements.  
\begin{example}\label{ex: sectioned fibration}
Let $p\colon S^2 \vee S^2 \to S^2$ denote the projection onto the first summand.  Convert $p$ into a fibration, to obtain a fibre sequence
$$\xymatrix{ X \ar[r]^-{j} & S^2 \vee S^2 \ar[r]^-{p} & S^2.}$$
Then $p$ admits the obvious section $i_1\colon S^2 \to S^2 \vee S^2$ which is just inclusion into the first summand.  Also, the fibration is not trivial, even after rationalization, since $H^*(S^2 \vee S^2; \Q)$ does not split as a tensor product $H^*(S^2; \Q) \otimes H^*(X;\Q)$.  Next, note that $X$ has infinitely many non-zero rational homotopy groups since, from the long exact sequence in rational homotopy, we have  $\pi_i(X_\Q) \cong \pi_i((S^2 \vee S^2)_\Q)$ for $i \geq 4$, and $S^2 \vee S^2$ has infintely many  non-zero rational homotopy groups.  
 To see that  $X$ has no rational Gottlieb elements,   we apply the  mapping theorem  for rational L.S. category.    Note that that the fibre inclusion $j \colon X \to S^2 \vee S^2$  induces an injection of rational homotopy groups.   Thus,  implies $\cat0(X) \leq \cat0( S^2 \vee S^2) = 1$    \cite[Th.28.6]{FHT}.  Now  $X$  is $\pi$-infinite   and so it follows that $X$ has the rational homotopy type of a wedge of at least two spheres.  Thus, the homotopy Lie algebra $\pi_*(\Omega X_\Q)$ is a free Lie algebra on at least two generators.  It follows directly that $X_\Q$ has no  nontrivial Gottlieb elements; for recall    a Gottlieb element $\alpha \in \pi_n(X_\Q)$ has vanishing Whitehead products $[\alpha, \beta] = 0$ for all $\beta \in \pi_q(X_\Q).$  
\end{example}   
The rationalization of the fibration of \exref{ex: sectioned fibration} cannot be a universal  fibration.  We  argue as follows.   First   by  a Serre spectral sequence argument,  $H^i(X;\Q) \not= 0$ for $i \geq 2$. Since $\mathrm{cat}_0(X) = 1$,   $X_\Q$ is  an infinite wedge of spheres.  We prove:  

\begin{proposition} 
Suppose that $X$ has  the rational homotopy type of a wedge of infinitely many spheres. Then   $\B(X_\Q)$ is $\pi$-infinite.  \end{proposition}

\begin{proof}
We note that Gatsinzi has proved essentially this fact for $X_\Q$ a finite wedge (cf. \cite{Gat1}). 
  We use the model for $\B(X_\Q)$ described in \cite[VII.2]{Tanre} that derives from a Quillen model of $X$.  Since $X$ is a wedge of spheres, it has Quillen model a free Lie algebra $\L(V)$ with zero differential and $V$ a graded vector space of generators isomorphic to the (de-suspension of the) reduced homology of $X$.  We assume this homology is infinite-dimensional, and so write $V = \{ v_i \}_{i \in \N}$ with the generators $v_i$ written in order of increasing degree: thus $i < j \implies |v_i| \leq |v_j|$.   Now the model for $\B(X_\Q)$ is a DG Lie algebra written as
$$(s\L(V) \oplus \Der\  \!  \L(V), \delta),$$
where $s\L(V)$ denotes the abelian Lie algebra on the vector space $\L(V)$ with degrees shifted up by one, $\Der\  \! \L(V)$ denotes the usual Lie algebra of derivations that increase degree, except that in degree $1$ we restrict to just the cycles. The differential $\delta$ restricts to the usual differential on $\Der\ \! \L(V)$ which, since we assume $X$ is a wedge of spheres, is zero here.  The only non-trivial differentials, therefore, are of elements from $s \L(V)$, where we have $\delta( sx) = \ad(x) \in \Der\ \! \L(V)$, for $x \in \L(V)$.  See \cite[VII.2(6)]{Tanre} for details.    Suppose that $v_n \in V$ is the first generator of degree strictly greater than that of $v_0$.  Then we define an infinite sequence of non-zero derivations $\theta_i \in \Der\ \!\L(V)$ for $i \geq n$ by setting $\theta_i(v_0) = v_i$ and $\theta_i = 0$ on all other generators, and then extending $\theta_i$ as a derivation to an element of $\Der\ \!  \L(V)$.  Then each $\theta_i$ is a $\delta$-cycle in   $s\L(V) \oplus \Der\ \! \L(V)$ that cannot be exact, since any boundary $\ad(x)$ will map $v_0$ to elements of bracket length at least two in $\L(V)$.  The homology of the Quillen model $(s\L(V) \oplus \Der\ \! \L(V), \delta)$ gives the   homotopy of $\B(X_\Q)$, which therefore is non-zero in infinitely many degrees.
\end{proof}

Our main result in this section is  a further example with the same  conclusion as Proposition \ref{S2}.

\begin{theorem} \label{S3}
Suppose  $X$ is a simply connected space with $$\B(X_\Q) \simeq \left(S^3 \vee \cdots \vee S^3\right)_\Q   \hbox{\, (a wedge of two or more copies of $S^3$).}$$ Then the universal fibration $X_\Q \to  E_{X_\Q} \to \B(X_\Q)$    has a section. 
\end{theorem}

We prove a preliminary result on  the homology of derivations of  a Sullivan minimal model  having a free factor generated by an element of even degree.  We say  an element $x$ in  a graded Lie algebra $L$ is a {\em zero-divisor} of $L$ if there exists $y \in L$ with $y \not\in \langle x \rangle$ such that $[x, y] = 0.$

\begin{lemma}\label{Z}
Let $(\land V, d)$ be a  Sullivan minimal model admitting a DG algebra factorization
$(\land V, d) \cong (\land (y), 0) \otimes (\land W, d)$ for  some $y \in V^{2m}.$    Let $\theta \in \Der_{2m}(\land V)$ be dual to $y.$   If  the homology class $[\theta]$ is not a zero divisor in $H_*(\Der (\land V), D)$, then  $$H_{\geq 2m}(\Der(\land V), D) = \langle [\theta] \rangle.$$   
\end{lemma}

\begin{proof}
Suppose that we have a derivation cycle $\alpha \in \Der_q(\land V)$ with $q \geq 2m$  representing a class $[\alpha] \in H_{\geq 2m}(\Der(\land V), D)$ in the complement of  $\langle [\theta] \rangle$.  We will show that  if $[\theta]$ is  not a zero divisor then  $\alpha$ is a $D$-boundary. 

Without loss of generality, we assume that $\alpha(y) = 0,$ for otherwise we may replace $\alpha$ with $\alpha - \theta.$       Given  $\chi \in \land W$, write 
$$\alpha(\chi) = \alpha_0(\chi) + y\alpha_1(\chi) + \cdots + y^r \alpha_r(\chi) + \cdots,$$
which defines derivations $\alpha_i$ of $\land W$ for $i \geq 0$.  Since $\alpha$ is a $D$-cycle, and $y$ is a $d$-cycle, and the differential $d$ restricts to one of $\land W$, it follows that each $\alpha_i$ is a $D$-cycle.  Notice that the sum here is locally finite, namely only finitely many terms are non-zero for a fixed $\chi$.  We will show inductively that we have a sequence of derivations $\{\eta_i\}_{i \geq 0}$ of $\land W$, for which $D(\eta_i) = \alpha_i$, and it follows that the (locally finite) sum 
$$ \eta(\chi) = \eta_0(\chi) + y\eta_1(\chi) + \cdots + y^r \eta_r(\chi) + \cdots,$$
has $D(\eta) = \alpha$.  Induction starts with $i = -1$, where there is nothing to prove.  For some $r \geq -1$, assume that we have derivations $\{\eta_i\}_{0 \leq i \leq r}$ with $D(\eta_i) = \alpha_i$.  Then $D(y^i\eta_i) = y^i\alpha$, and
$$\alpha - D\left( \sum_{i=0}^{r}\ y^i \eta_i \right)(\chi)  = y^{r+1}\alpha_{r+1}(\chi) + y^{r+2}\alpha_{r+2}(\chi) + \cdots .$$
%
%(The induction starts case, with $i = -1$, simply means that we have the original $\alpha$ with none of the $\eta_i$, and the induction step will construct $\eta_0$.)   
Write $\beta(\chi) = \alpha_{r+1}(\chi) + y\alpha_{r+2}(\chi) + y^{2}\alpha_{r+3}(\chi)+ \cdots,$ so that
$$\alpha - D\left( \sum_{i=0}^{r}\ y^i \eta_i \right)  = y^{r+1} \beta.$$
Since the left-hand side here is a $D$-cycle, it follows that so too is $\beta$ a $D$-cycle.  Now use $\beta$ to construct the derivation
$$\widehat{\beta} = \beta - y\, [\theta, \beta] + \frac{y^2}{2!}\, \big[ \theta, [\theta, \beta]\big] + \cdots +  \frac{y^n}{n!}\, \mathrm{ad}^n(\theta)(\beta) + \cdots.$$
Since $y$ is a $d$-cycle, and $\theta$ and $\beta$ are both $D$-cycles, each term in this (locally finite) sum is a derivation that is a $D$-cycle.  Hence, $\widehat{\beta}$ is a 
$D$-cycle.  But observe that we have 
$[\theta, \widehat{\beta}] =   ([\theta, \beta] - 1\cdot [\theta, \beta]) + (-  y\, \big[ \theta, [\theta, \beta]\big] + y\, \big[ \theta, [\theta, \beta]\big] ) + \cdots = 0.$
Then our assumption that $\theta$ is not a zero divisor implies that  $\widehat{\beta} = D(\eta)$ for some $\eta \in \Der_{q+1}(\land V)$.   However, we have $\widehat{\beta} = \beta_0$.  This follows by writing out terms in $\widehat{\beta}(\chi)$, as
$$
\begin{aligned}
\widehat{\beta}(\chi) &= \beta(\chi) - y\, [\theta, \beta] (\chi) + \cdots \\
&= \beta(\chi) - y\, \theta\circ \beta (\chi) + \cdots + \frac{y^n}{n!}\,\theta^n\circ \beta(\chi) + \cdots \\
&= (\alpha_{r+1}(\chi) + y\alpha_{r+2}(\chi) + y^{2}\alpha_{r+3}(\chi)+ \cdots)  \\ &\hbox{\hskip.75truein}  - y\, \theta(\alpha_{r+1}(\chi) + y\alpha_{r+2}(\chi) + y^{2}\alpha_{r+3}(\chi)+ \cdots)\\
&\hbox{\hskip.75truein} + \cdots + \frac{y^n}{n!}\,\theta^n(\alpha_{r+1}(\chi) + y\alpha_{r+2}(\chi) + y^{2}\alpha_{r+3}(\chi)+ \cdots) + \cdots \\
&= (\alpha_{r+1}(\chi) + y\alpha_{r+2}(\chi) + y^{2}\alpha_{r+3}(\chi)+ \cdots) \\ &\hbox{\hskip.75truein} - ( y\alpha_{r+2}(\chi) + 2y^{2}\alpha_{r+3}(\chi)+ 3y^{3}\alpha_{r+4}(\chi)+\cdots)\\
&\hbox{\hskip.75truein} + \cdots + ( y^n\alpha_{r+1+n}(\chi) + {{n+1}\choose{n}}y^{n+1}\,\alpha_{r+1+n+1}(\chi)+ \cdots) + \cdots.
\end{aligned}
$$
Notice above, in passing from the first to the second lines, we replaced $\mathrm{ad}^n(\theta)(\beta)(\chi)$ with $\theta^n\circ \beta(\chi)$ since $\theta(\chi) = 0$.
A careful tallying of the terms in this last expression reveals that it consists of a sum of terms $y^n\alpha_{r+1+n}(\chi)$ for $n \geq 0$ and, for $n \geq 1$, the coefficient of 
$y^n\alpha_{r+1+n}(\chi)$ is the alternating sum
$$1 - n + {{n}\choose{2}} + \cdots + (-1)^t\, {{n}\choose{t}} + \cdots + (-1)^n.$$
But this is zero, since it is the value of $(1-x)^n$ when $x = 1$.  That is, we have 
$$D(\eta)(\chi) =  \widehat{\beta}(\chi) = \alpha_{r+1}(\chi).$$
But this means that we have $D(\eta)  = \alpha_{r+1}$, as derivations of $\land W$, and setting $\eta = \eta_r$ completes the inductive step.  Notice that $\alpha_r$ is of degree
$|\alpha_r| = |\alpha| - 2mr,$
and $|\eta_r| = |\alpha_r| +1$, so for a fixed $\chi$   only finitely many $\eta_r(\chi)$ can be non-zero, for degree reasons. It follows from the induction that,
with 
$$\eta = \sum_{i\geq 0}\ y^i \eta_i,$$
we have $D(\eta) = \alpha$.   
\end{proof}

%\begin{remark}
%We supposed that $[\theta]$ is not a zero divisor in the above, but the same argument may be used to draw conclusions with much weaker hypotheses. For instance, since $\widehat{\beta}$ and $\alpha$ have the same degree, it would be sufficient to suppose that $\big[ [\theta], \omega\big] = 0 \implies \omega = 0 \in H(\Der \land V)$ for all $\omega$ of degree $N = |y|$ to conclude that  $H_N(\Der \land V) = 0$.  We use some of  these extensions in drawing consequences in \secref{sec: main result}.
%\end{remark}

\begin{proof}[Proof of Theorem \ref{S3}]  
Suppose $X$ with Sullivan minimal model $(\land V, d)$ satisfies $\B(X_\Q) \simeq (S^3 \vee \cdots \vee S^3)_\Q.$  Let $p_{X_\Q}  \colon E_{X_\Q} \to \B(X_\Q)$ denote the universal fibre map for $X_\Q$.   Suppose    $(p_{X_\Q})_\sharp \colon \pi_3(E_{X_\Q}) \to \pi_3(\B(X_\Q))  $  is not surjective.   Then there is  a  Gottlieb element $y \in V^2$.      Since $(\land V, d)$ is simply connected,  $dy = 0$.  Thus we have a factorization    $(\land V, d) \cong  (\land(y), 0)  \otimes (\land W, d)  $  for $W$ complementary to $\langle y \rangle$ in $V$.  We may   apply Lemma \ref{Z}.  However, here we have $$H_*(\Der(\land V), D) \cong \pi_*\big( \Omega \B(X_\Q)\big)\cong \pi_*\big( \Omega (S^3 \vee \cdots \vee S^3)_\Q\big) $$ and the latter is the free graded Lie algebra generated in degree $2$ by at least two elements.  No element of degree $2$ (or any other degree)    can be a zero divisor  and yet neither can we have $H_{\geq 2}(\Der (\land V), D)$ one-dimensional (indeed, $H_{\geq 2}(\Der(\land V), D)$ is infinite-dimensional).  We have a contradiction.

We may assume  $(p_{X_\Q})_\sharp \colon \pi_3(E_{X_\Q}) \to \pi_3(\B(X_\Q))  $  is   surjective.    Then,  for each summand $S^3_\Q$ involved in  the wedge, we may lift the inclusion   $S_\Q^3 \to (S^3 \vee \cdots \vee S^3)_\Q$ through $p_{X_\Q}$ to a map $S_\Q^3 \to E_{X_\Q}$.  Assembling  these liftings gives a lifting of the identity map of $\B(X_\Q) \simeq (S^3 \vee \cdots \vee S^3)_\Q$ through $p_{X_\Q}$ and so a section of the universal fibration.   
\end{proof}

 We conclude with a homotopy-theoretic proof of Theorem \ref{terminal} using the universal fibration for rational spaces.       \begin{proof}[Alternate Proof of Theorem \ref{terminal}]
 Given a space $X$ with a Gottlieb element $\alpha \in \pi_{2n}(X)$ we show that there is a corresponding terminal element $y \in V^{2n}$ in a Sullivan minimal model $(\land V, d)$ for $X$.   
 Let $F \colon S^{2n} \times X \to X$ be the affiliated map for $\alpha.$  
  The rationalization of   $F$ is a map $F_0 \colon S^{2n}_\Q  \times   X_\Q  \to X_\Q$ which has adjoint $\beta_0 \colon S^{2n}_\Q \to \aut(X_\Q).$  From the isomorphisms 
  $$\pi_{2n+1}(\B(X_\Q)) \cong \pi_{2n}(\Omega \B(X_\Q)) \cong \pi_{2n}(\aut(X_\Q))$$
we may view $\beta_0$ as a class $\beta_0 \colon S^{2n+1}_\Q  \to \B(X_\Q)$.  The pullback of the universal fibration   $X_\Q \to E_{X_\Q} \to \B(X_\Q)$ by $\beta_0$  is a fibration of the form $X_\Q \to E_\Q \to S^{2n+1}_\Q$ with connecting homomorphism satisfying $\partial_\sharp(\iota_{2n+1}) = \alpha_0$ where $\iota_{2n+1} \in \pi_{2n+1}( S^{2n+1}_\Q)$ is the fundamental class and $\alpha_0 \in \pi_{2n}(X_\Q)$ image of $\alpha$ under rationalization.      Stepping back one stage in the Puppe sequence gives a fibration $\Omega S^{2n+1}_\Q  \to X_\Q \to E_\Q$.  Now observe $\Omega S^{2n+1}_\Q = K(\Q, 2n)$. 
The relative Sullivan model for this fibration is thus of the form $$(\land V', d') \rightarrow (\land(V' \oplus \langle y \rangle), D)    \to (\land(y), 0).$$
Since $y \in V^{2n}$ is a Gottlieb element, using \cite[Pro.15.13]{FHT} we can see that  middle term  is a minimal Sullivan model for $X_\Q$ having  terminal homotopy  element $y$ corresponding to the class $\alpha$.  
 \end{proof}

%\nocite{*}

\bibliographystyle{amsplain}
\bibliography{Baut}

 \end{document}